\documentclass[12pt]{amsart}
\usepackage{color,latexsym,amssymb,amsmath,amsthm}
\usepackage{comment}
\usepackage{graphicx}
\definecolor{orange}{rgb}{1,0.5,0}

\newtheorem{thm}{Theorem}

\theoremstyle{definition}

\newcommand{\E}{{\rm E}}

\newcommand{\R}{\mathbb{R}}

\title{Curvature from Graph Colorings}
\author{Oliver Knill}
\date{October 5, 2014}

\address{
        Department of Mathematics \\
        Harvard University \\
        Cambridge, MA, 02138
        }


\begin{document}
\maketitle

\begin{abstract}
Given a finite simple graph $G=(V,E)$ with chromatic number $c$ and chromatic polynomial $C(x)$.
Every vertex graph coloring $f$ of $G$ defines an index $i_f(x)$ satisfying the Poincar\'e-Hopf theorem 
\cite{poincarehopf} $\sum_x i_f(x)=\chi(G)$. As a variant to the index expectation result 
\cite{indexexpectation} we prove that $\E[i_f(x)]$ is equal to curvature $K(x)$ satisfying Gauss-Bonnet 
$\sum_x K(x) = \chi(G)$ \cite{cherngaussbonnet}, where
the expectation is the average over the finite probability space containing the $C(c)$ 
possible colorings with $c$ colors, for which each coloring has the same probability. 
\end{abstract}

\section{Introduction}

Differential geometric ideas were used in graph coloring problems early on.
Wernicke \cite{Wernicke} promoted already Gauss-Bonnet type ideas to the four color
problem and used what we today call curvature to graph theory.
The discharging method introduced by Heesch \cite{Heesch} 
was eventually used by Appel and Haken  \cite{AppelHaken}
to prove the four color theorem. This method is conceptionally related to geometric evolution methods: 
deform the geometry until a situation is reached which can be understood by classification.
Birkhoff \cite{BirkhoffPolynomial} and later Tutte attached polynomials to graphs 
which relate modern topological invariants like the Jones polynomial. Cayley already 
geometrized the problem by completing and modifying planar graphs to make them 
what we would call cubic so that a planar graph looks like discretization of a two 
dimensional sphere which by Gauss Bonnet has Euler characteristic $\sum K(x) = 2$ forcing the
existence of vertices with positive curvature. Studying the positive curvature parts is essential in the proof 
of the four color theorem for planar graphs \cite{FritschFritsch,Barnette,RobinWilson}. \\

Probabilistic methods in geometry are part of integral geometry as pioneered by mathematicians like
Crofton, 
Blaschke 
and Santal\'o 
\cite{Blaschke,Santalo,KlainRota,Nicoalescu}. 
This subject is known also under the name geometric probability theory
and has been used extensively in differential geometry. For example by Chern in the form of kinematic 
formulae \cite{Nicoalescu} or Milnor in proving total curvature estimates \cite{MilnorKnot,Spivak1999}
using the fact that curvature $K(x)$ of a space curve is an average of 
indices $i_f(x)$ (at minima of $f$ on the curve)  using a probability space of linear functions $f$. 
Banchoff used integral geometric methods in \cite{Banchoff67,Banchoff70} 
and got analogue results for polyhedra and surfaces. We have obtained similar results in 
\cite{indexexpectation} for graphs. Integral geometry is also used in other areas and is an applied topic:
the Cauchy-Crofton formula for example expresses the length of a curve as the average number of intersection 
with a random line. The inverse problem to reconstruct the curve from the number of 
intersections in each direction is a Radon transform problem and part of tomography, a concept which has
also been studied in the concept of graph theory already \cite{IntegralGeometryGraphs}.  \\

In this note we look at an integral geometric topic when coloring finite simple graphs $G=(V,E)$.
In full generality, the scalar curvature function (\ref{curvature})
satisfies the discrete Gauss-Bonnet formula $\sum_x K(x) = \chi(G)$. In \cite{cherngaussbonnet},
the focus had been on even dimensional geometric graphs for which $K(x)$ is a discretized Euler curvature,
a Pfaffian of the curvature tensor (\cite{Cycon} chapter 12).
For odd-dimensional geometric graphs, it was shown in \cite{indexexpectation} using integral geometric methods 
that the curvature is always constant zero \cite{indexformula}, a fact already pioneered in piecewise linear 
polytop situations \cite{Banchoff67}. 

\section{Index expectation}

Given a function $f:V \to \R$ on a vertex set $V$ of a finite simple graph, 
let $S^-_f(x)$ denote the subgraph generated by $\{ y \in S(x) \; | \; f(y)<f(x) \; \}$,
where $S(x) = \{ y \in V \; | \; (x,y) \in E \}$ generates a graph called 
the unit sphere of the vertex $x$.  In \cite{indexexpectation}, we averaged the index 
$$  i_f(x) = 1-\chi(S_f^-(x)) $$
over a probability space functions $f: V \to [-1,1]$ in such a way that for  a fixed vertex, 
the random variables $f(x)$ are all independent and uniformly distribution in $[-1,1]$. 
We showed that its expectation is the curvature
\begin{equation} \label{curvature}
 K(x) = \sum_{k=0}^{\infty} (-1)^k \frac{V_{k-1}(x)}{k+1} \; , 
\end{equation}
where $V_k(x)$ is the number of $K_{k+1}$ subgraphs in the sphere 
$S(x)$ at a vertex $x$ and $V_{-1}(x)=1$. \\

We know that 
$$  \sum_x i_f(x) = \chi(X) $$ 
if $f$ is injective. This is Poincar\'e-Hopf \cite{poincarehopf}. 
It can be proven by induction: add a vertex $x$ to a graph and extend $f$ so 
that the new point is the maximum of $f$. This increases the Euler characteristic by 
$\chi(B(x)) -\chi(S_f^-(x)) = 1-\chi(S_f^-(x))$ using 
the general formula $\chi(A \cup B) = \chi(A) + \chi(B) - \chi( A \cap B)$
used in the case where $A$ is the graph without the new point $B(x)$ 
is the unit ball of the new point $x$ and $A \cap B$ is $S^-_f(x)$.
The expectation of Poincar\'e-Hopf recovers then the Gauss-Bonnet theorem \cite{cherngaussbonnet}
$$ \sum_{x \in V} K(x) = \chi(G) \; , $$
where $\chi(G)=\sum_{k=0}^{\infty} (-1)^k v_k$ is the Euler characteristic
of the graph defined as the super count of  the number $v_k$ of $K_{k+1}$ subgraphs of $G$. 
Gauss Bonnet also immediately follows from the generalized handshaking lemma 
$$  \sum_x V_{k-1}(x) = (k+1) v_k \; . $$ 
While \cite{poincarehopf} established Poincar\'e-Hopf only for injective functions,
the same induction step also works for locally injective functions: assume it is true for all graphs with 
$n$ vertices and functions $f$ which are locally injective, take a graph with $n+1$ vertices and chose
a vertex which is a local maximum of $f$ at $x$ then remove this vertex together with the connections to $S^-(x)$.
Again this reduces the Euler characteristic by $1-\chi(S^-(x))$. \\

The integral geometric approach to Gauss-Bonnet illustrates an insight attributed in \cite{Nicoalescu} to 
Gelfand and his school: {\it the main trick of classical integral geometry is the ``change of order of summation"}.
In our case, one summation happens over the vertex set $V$ of the graph, the other integration is performed
over the set of functions $f$ with respect to some measure. 
What can be proven in a few lines for general finite simple graphs can also be done for
general compact Riemannian manifolds. The integral geometric approach provides intuition what Euler 
curvature is: it is an expectation of the quantized index values or divisors given as Poincar\'e-Hopf 
indices which incidentally play a pivotal role also in Baker-Norine theory \cite{BakerNorine2007} and
especially in higher dimensional versions of this Riemann-Roch theorem.
That also in the continuum, curvature is the expectation of 
index density is intuitive already in the continuum for manifolds $M$: negative index values are more likely to occur at places
with negative curvature; there is still much to explore however. We need to investigate in particular what happens if we take
the probability space $M$ itself and chose for every $x \in M$ and some fixed time $\tau$ the heat signature function
$f(y) = [e^{-\tau L}]_{xy}$, if $L=d d^* + d^* d$ is the Hodge Laplacian of the Riemannian manifold. 
The question is then how the expectation $K_{\tau}(y) = \E[i_f(y)]$ of the index 
density of these functions is related to the classical Euler curvature \cite{eveneuler}. This is interesting
because it would allow us to work with a small (finite dimensional) probability space. 
The smallness of the probability space leads us to the topic covered here. \\

We look at one of the smallest possible probability spaces for which an index averaging result can hold
for graphs. It turns out that we do not need injectivity of the functions $f$ at all. All we need is 
local injectivity. In other words, we can restrict the probability space of functions $f$ to {\bf colorings}. 
This reduces the size of the probability space considerably
similarly as in the continuum, the choice of linear functions reduced the probability space from an infinite
dimensional class of Morse functions to a finite dimensional space of linear functions induced onto the 
manifold from an ambient Euclidean space. 
While Poincar\'e-Hopf and the index averaging result hold also for all $n!$ injective functions, 
the probability space of $c$ colorings has considerably
less elements. If $C(x)$ is the chromatic polynomial of the graph introduced in 1912 by Birkhoff in the form of a
determinant \cite{BirkhoffPolynomial}, then the probability space has $C(c)$ elements.  \\

For the octahedron for example, a geometric two-dimensional graph $G$ with Euler characteristic $\chi(G)=2$ and
chromatic number $c(G)=3$, the probability space of colorings has $6$ elements only. We can see this in Figure~(\ref{fig5}).

\begin{figure}
\scalebox{0.45}{\includegraphics{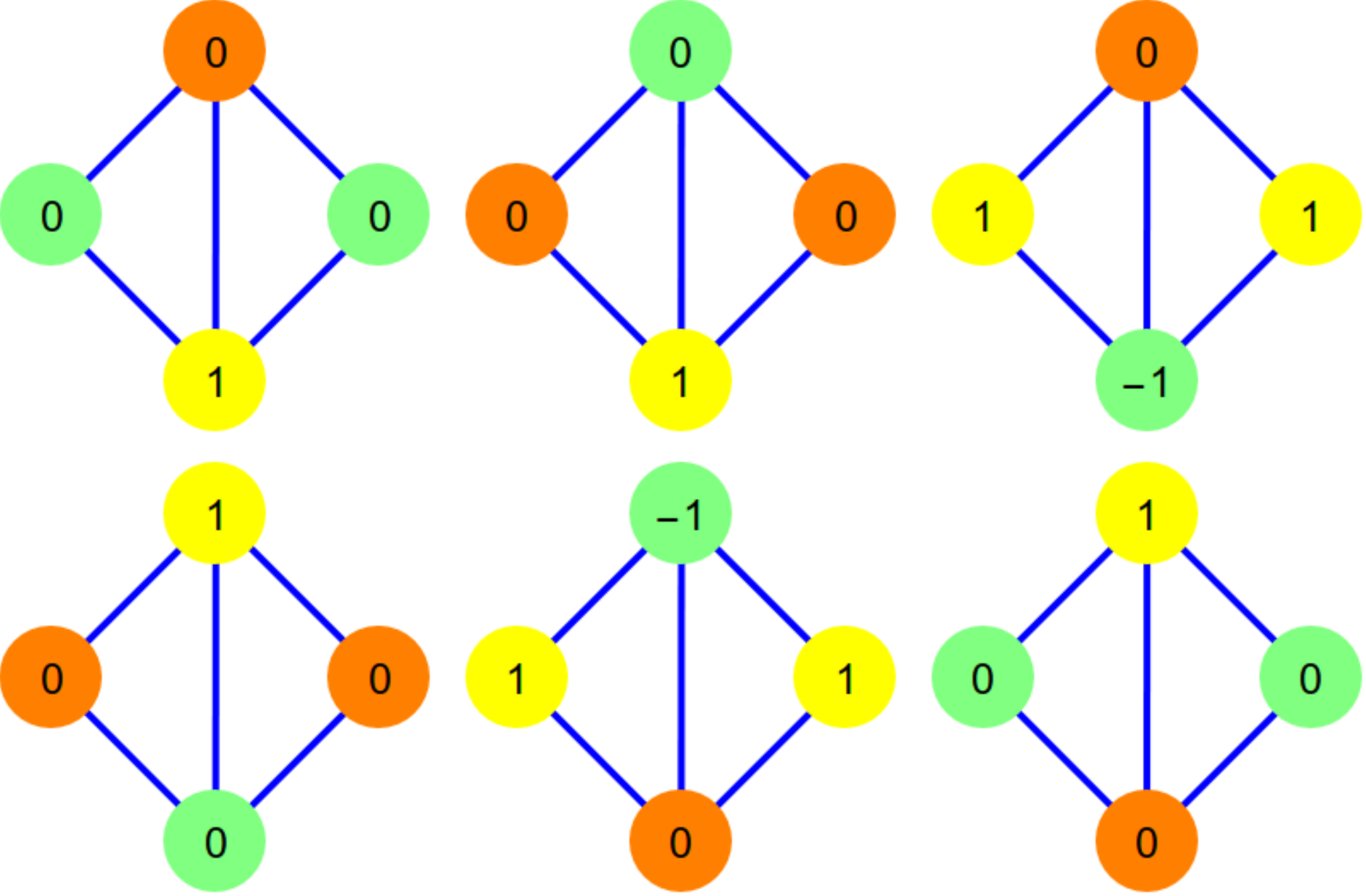}}
\caption{
We see all 3-colorings of the diamond graph. The index values are marked on the vertices.
The probability space has only $C(3)=6$ elements. 
The average index is the curvature $K = (1/6,1/3,1/6,1/3)$ of $G$. 
}
\end{figure}

\begin{figure}
\scalebox{0.45}{\includegraphics{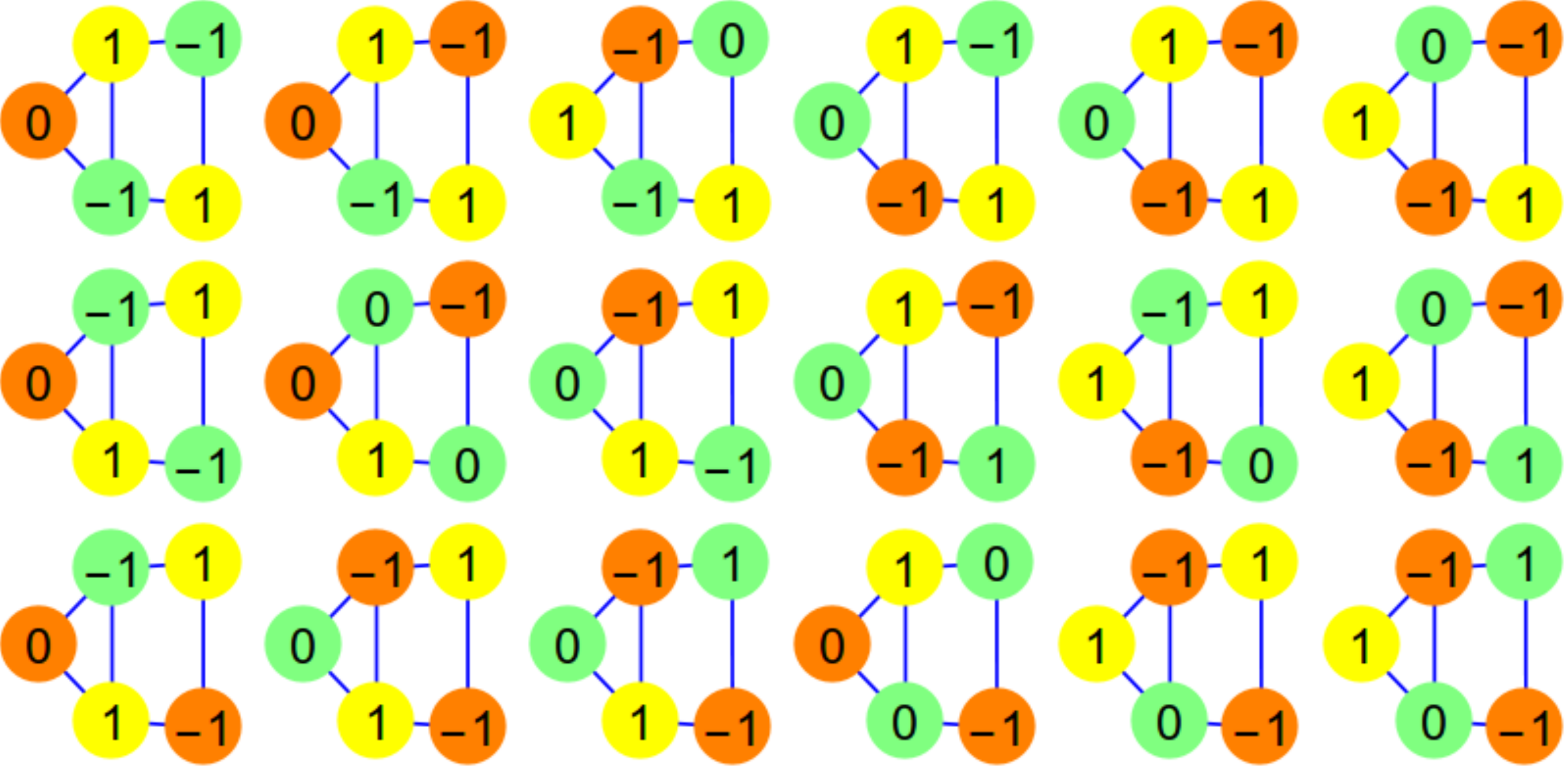}}
\caption{
All 3-colorings of the house graph and the index values. The probability space
has 18 elements. The average index is the curvature $K = (0,0,-1/6,-1/6,1/2)$
which adds up to zero. 
}
\end{figure}

\begin{figure}
\scalebox{0.45}{\includegraphics{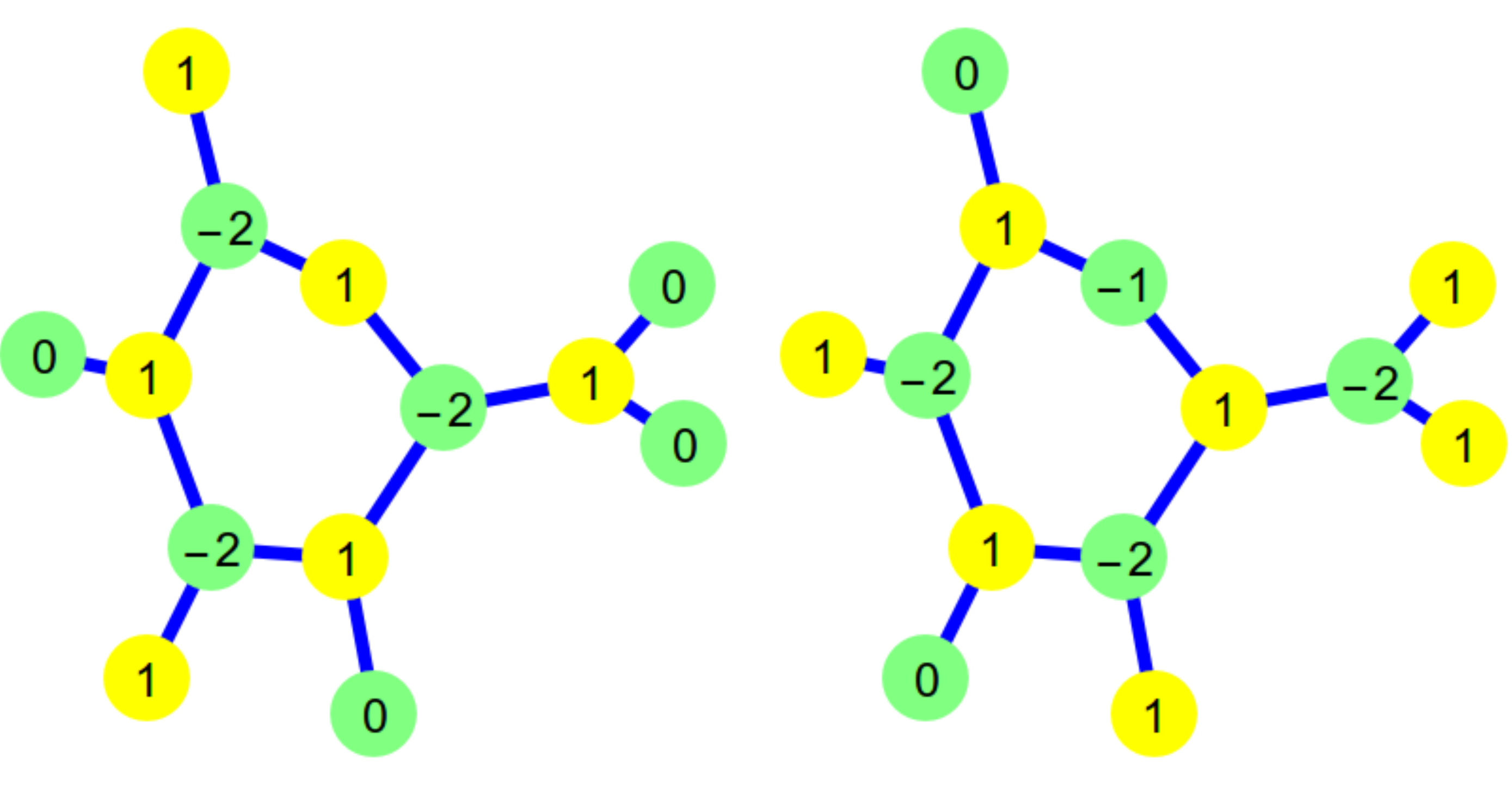}}
\caption{
The figure shows all 2-colorings of the Cytosine molecule graph. 
The probability space has only $C(2)=2$ elements. The average index is the curvature $K$.
It takes values $1/2, -1/2$ or $0$. The zero curvature appears at the nitrogen
atom which only bounds 2 carbon atoms. The hydrogen and oxygen atoms have
curvature $1/2$. The negative curvatures appear on carbon and hydrogen atoms
with 3 neighbors. 
}
\end{figure}

\begin{figure}
\scalebox{0.40}{\includegraphics{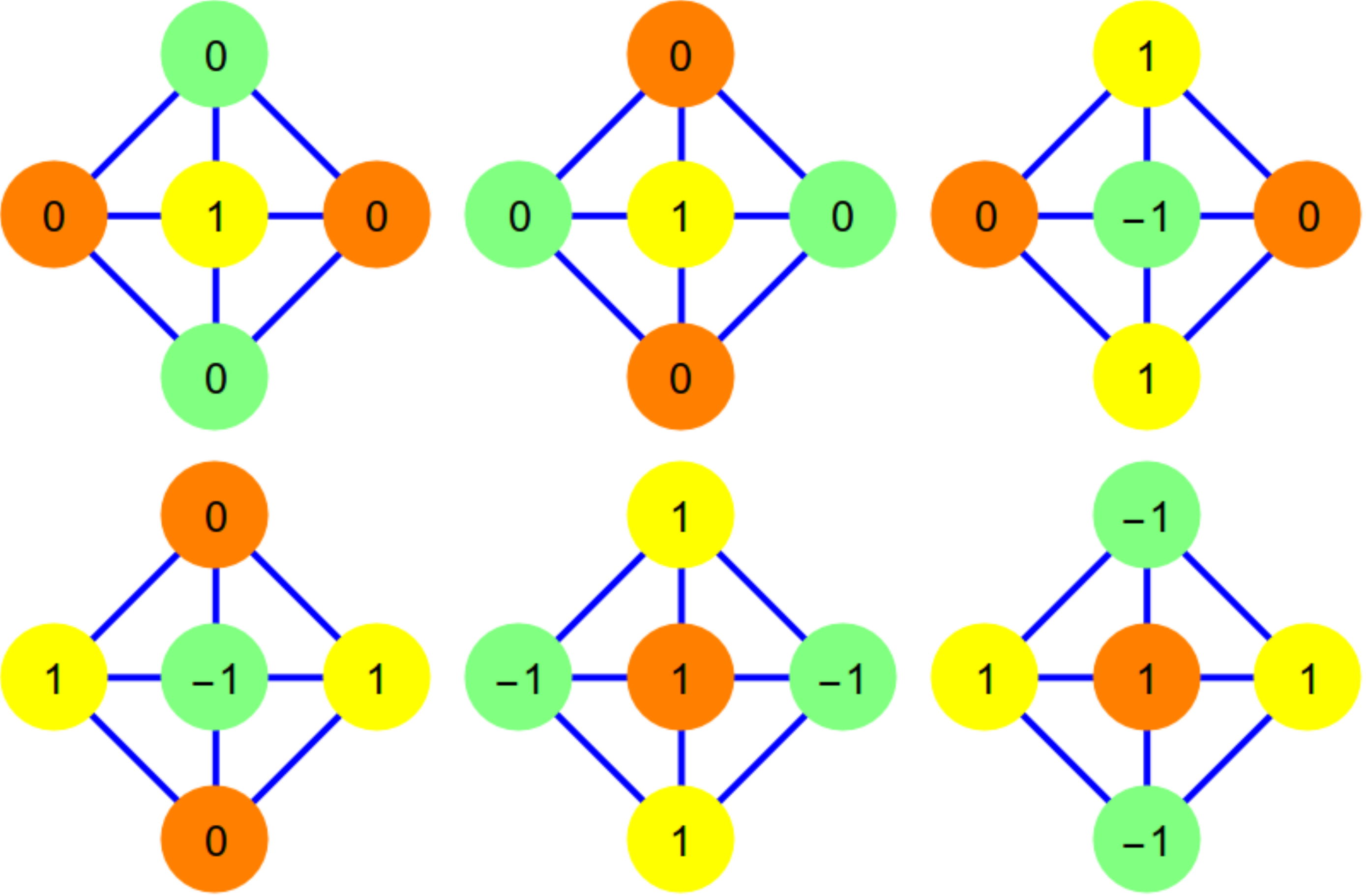}}
\caption{
All 3-colorings of the Boron Carbide molecule $CB_4$ graph which is just a wheel graph. 
The probability space has only $C(3)=6$ elements. The average index is the curvature 
$K$ which only takes values $1/3$ at the central carbon atom and $1/6$ on the 
peripheral boron atoms. 
}
\end{figure}

\begin{figure}
\scalebox{0.45}{\includegraphics{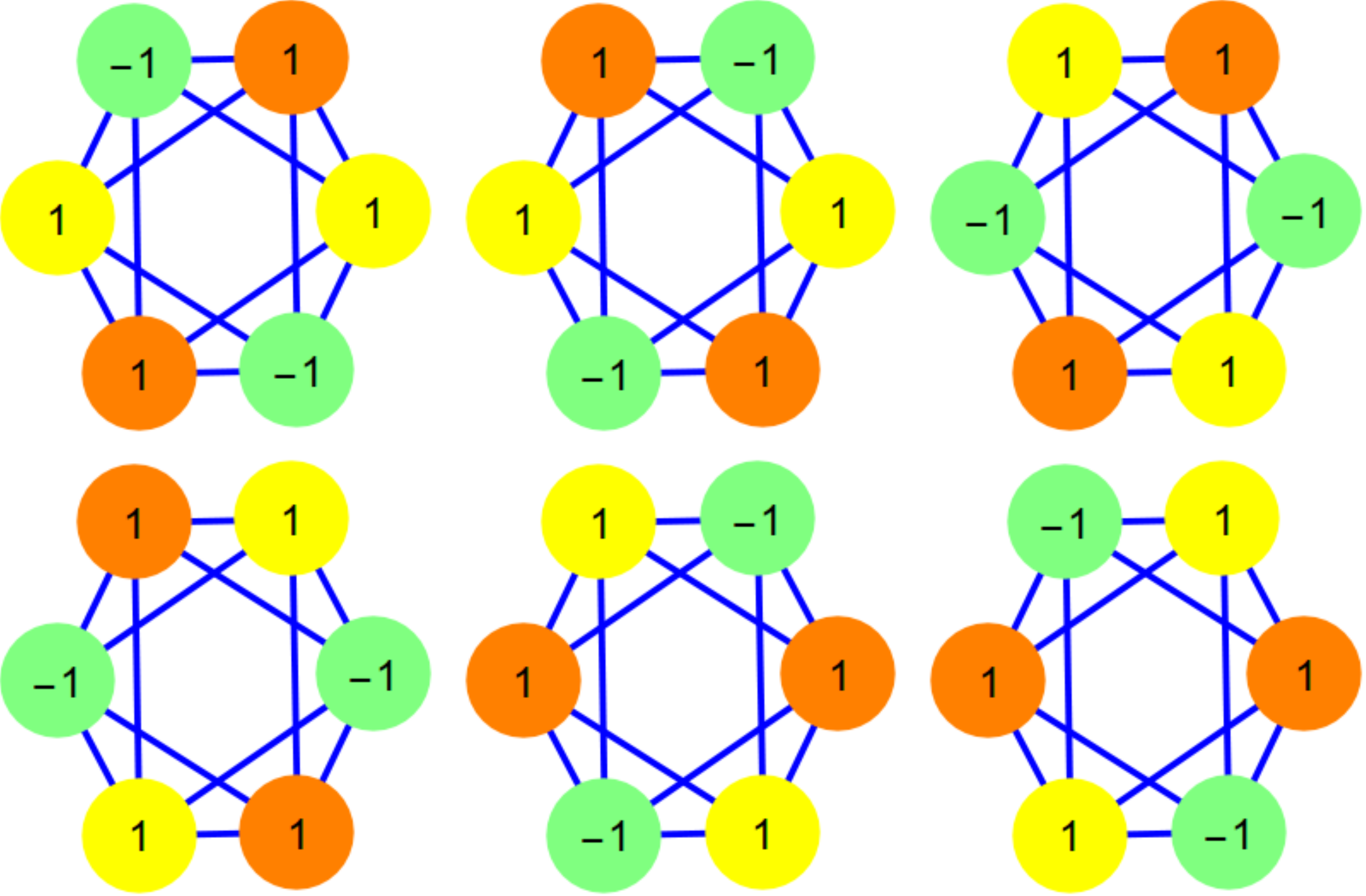}}
\caption{
All 3-colorings of the octahedron graph. 
The probability space has only $C(3)=6$ elements. The average index is the curvature
$K$ which takes the value $1/3$ on each vertex. The total curvature is $2$ by Gauss-Bonnet. 
}
\label{fig5}
\end{figure}

\begin{thm}[Color Index Expectation]
\label{mainresult}
Let $\Omega_c(G)$ denote the finite set of all graph colorings of $G$ with $c$ colors, 
where $c$ is bigger or equal than the chromatic number. Let $\E$ be the expectation
with respect to the uniform counting measure on $\Omega_c$. Then 
$$  \E[i_f(x)] = K(x) \; . $$
\end{thm}

\begin{proof}
The proof mirrors the argument for continuous random variables \cite{indexexpectation}. 
It is simpler as we do have random variables with a discrete distribution.
Let $V_k$ denote the number of $k$-dimensional 
simplices in $S(x)$ and $V_k^-$ the number of $k$-dimensional simplices in $S^-_f(x)$. 
Given a vertex $x$, the event 
$$  A = \{ f \; | \; f(x)>f(y), \forall y \in V \; \} $$ 
has probability $1/(k+2)$.  
The reason is that the symmetric group of color permutations
acts as measure preserving automorphims on the probability space
of functions implying that for any $f$ which is in $A$ there are $k+1$ functions which 
are in the complement implying that $A$ has probability $1/(k+2)$. This implies  
$$   \E[V_k^-(x)] = \frac{V_k(x)}{(k+2)}  \; $$
which is the same identity we also had in with the continuum probability space.  
Therefore,
\begin{eqnarray*}
   \E[ 1-\chi(S^-(x)) ] &=& 1-\sum_{k=0}^{\infty}  (-1)^k \E[V_k^-(x)] 
                         = 1+\sum_{k=1}^{\infty} (-1)^k \E[V_{k-1}^-(x)]  \\
                        &=& 1+\sum_{k=1}^{\infty} (-1)^k \frac{V_{k-1}(x)}{(k+1)} 
                         =  \sum_{k=0}^{\infty} (-1)^k \frac{V_{k-1}(x)}{(k+1)} = K(x)   \; .
\end{eqnarray*}
\end{proof}

For trees for example where $c=2$, there are only $C(2)=2$ colorings. At the 
leaves of the tree, we have either the index $0$ or $1$, where minima $x$ always have 
the index value $i_f(x) = 1$. 
The average index is therefore $1/2$ at the leaves. On an interior vertex $x$ with 
index $d(x)$ we have either index $1-0$ if the vertex is a minimum or $1-d(x)$ if it is 
a maximum. The curvature there is then $(1+(1-d(x))/2=(2-d(x))/2$. More examples 
can be seen in the figures. 

\section{Open ends} 

A) We can think of a graph coloring with $c$ colors as a
gauge field using the finite cyclic group $Z_{c}$. Permuting colors is a symmetry or
gauge transformation. Given a coloring $f$ and a fix of direction for every edge, we can look at the gradient
$df = \nabla f$, which is now just realized as a finite $v_0 \times v_1$ matrix from the module $Z_c^{v_1}$
to $Z_c^{v_0}$. With the transpose $d^*$ which is a discrete divergence, the Laplacian $L=d^* d$ is independent
of the directions chosen on the edges. When looking at such colorings, we get a probability space
of gradient vector fields. If $c$ is prime, we could look at the eigenvalues of $L$ in an extension field
of the Galois field $Z_c$. Looking at the situation where $c$ is the chromatic number of $G$ is natural. One can 
now ask spectral inverse problems with respect to such ``color spectra" or look at random walk 
defined by the Laplacian $L$. Dynamical systems on finite group-valued gauge fields
are special cellular automata \cite{Wolfram86}.
Especially interesting could be to study the discrete Markoff process $f \to L^n f$, 
which is now a cellular automaton. \\

B) An interesting problem is to characterize graphs with minimal chromatic number in 
the sense that the {\bf chromatic richness} $C(c)/c!$ is equal to $1$. Graphs which satisfy 
$C(c)=c!$ can be colored in a unique way modulo color permutations. We call them {\bf chromatically poor graphs}. 
Examples are trees (with $c=2$), the octahedron (with $c=3)$, the complete graphs $K_n$ (with $c=n$), 
wheel graphs $W_{2n}$ with an even number of spikes (with $c=3$) or cyclic graphs $C_{2n}$ or more generally 
any bipartite graph with $c=2$.
Since for cyclic graphs $C_{2n+1}$, or wheel graphs $W_{2n+1}$ we have $C(c)/c!=(4^n-1)/3$, the richness 
of a graph is expected to be exponentially large in $v_0=|V|$ in general. We have done some statistics with the
richness function on Erd\"os-Renyi probability spaces of graphs. It appears that a limiting distribution 
could appear. But this is very preliminary as experiments with larger $n$ are costly. 
For simply connected geometric graphs, we only need satisfy a local
degree condition attributed to Heawood (see Theorem 2.5 in \cite{SaatyKainen} or Theorem 9.5 
in \cite{ChartrandZhang}).
Related is the conjecture of Steinberg from 1975 which claims that every planar graph 
without 4 and 5 cycles must be 3 colorable. More generally, Erd\"os asked for which $k$ the 
exclusion of cycles of length $\{4,\dots ,k\}$ renders the chromatic number $c\leq 3$. 
The record is $k=7$ \cite{BGRS}.  While graphs without cycles of length larger than $3$
have $c \leq 3$ already if two triangles are present, such a graph can not be chromatically 
poor any more.\\

C) What does the standard deviation of the random variable $f \to i_f(x)$ tell about the graph? 
More generally, what is the meaning of the moments $\E[i_f(x)^n]$? 
Like curvature $K(x)$, these are scalar functions on vertices.
The standard deviation is a scalar field which  depends on the number of colorings. 
For the wheel graph $C_4$ for example which has chromatic number $3$ and chromatic polynomial 
$C(x) = 14x - 31x^2 + 24x^3 - 8x^4 + x^5$, we measure the standard deviations \\
$\sigma(3)=(0.942,0.687,0.687,0.687,0.687)$, $C(3)=6$, \\
$\sigma(4)=(0.816,0.645,0.645,0.645,0.645)$, $C(4)=72$, \\
$\sigma(5)=(0.738,0.613,0.613,0.613,0.613)$, $C(5)=420$. \\

D) A general problem in differential geometry is to determine to which extent curvature
determines geometry \cite{BergerPanorama}. One can now ask for graphs, to what extend the moments 
$\sum_{x \in V} K(x)^k$ determine the graph up to isomorphism. With the index functions
$i_f(x)$ one has even more data and one can ask to which extent the sequence of numbers
$$ a_k =\frac{1}{C(c)} \sum_{f \in \Omega_c} \sum_{x \in X} i_f^k(x) \;  $$
determines the graph $G$ up to isomorphism.  Of course, we have  $a_1 = \chi(G)$ already by 
Poincar\'e-Hopf. Also interesting are the moments
$$ a_k(x) =\frac{1}{C(c)} \sum_{f \in \Omega_c} i_f^k(x) \;  $$
for which we just have established $a_1(x)=K(x)$. 
One can ask the inverse problem of finding the geometry from the moments 
$K_n(x)=\E[i_f(x)^n]$ in any geometric setup. While for plane curves, $K_1$ alone
determines the curve up to isometry, for space curves (where total curvature is nonnegative),
this is already no more the case. The analogue of total curvature would be $\E[|i_f(x)|]$ 
which can be expressed with finitely many of the above moments for any graph. 
Still, it is conceivable that the distribution of the moment functions $K_n(x)$ is sufficient
to reconstruct the graph up to isomorphism. \\

E) Index averaging allows to see that curvature $K(x)$ is identically zero for
odd-dimensional geometric graphs \cite{indexformula}, 
finite simple graphs for which the unit spheres are discrete spheres, geometric $(d-1)$ dimensional
graphs which become contractible after removing one vertex. The key insight from \cite{indexformula} is
that the symmetric index $j_f(x) = (i_f(x) + i_{-f}(x))/2$ is related to the Euler characteristic of a
geometric $(d-2)$-dimensional graph. Establishing $j_f(x)$ to be zero for all $x$ implies curvature $K(x)$ 
is zero.  The same holds for Riemannian manifolds $M$ where it leads to the insight that 
Euler characteristic $\chi(M)$ for even dimensional manifolds is an average of two dimensional 
sectional curvatures as $\chi(M)$ is also an average of $\chi(N)$ where $N$ 
runs over a probability space of two-dimensional submanifolds $N$ of $M$.  As in the case of graphs, a probability measure $P$ 
on the set of scalar fields provides us with with curvature and defines geometry.
In physics, one would select a measure $P$ which is invariant under the wave equation and
refer to the Krylov-Bogolyubov theorem for existence.
The same can be done using the finite probability space considered here:
any coloring $f$ of a four dimensional geometric graph defines a two dimensional geometric graph. 
The finite probability space of graph colorings allows computations with rather small number of elements.
The natural measure on functions is the product measure. \\

F) Integral geometry naturally bridges the discrete
with the continuum. This was demonstrated early on in the proof of the F\'ary-Milnor theorem
\cite{MilnorKnot}, where a knot was approximated with polygons for which curvature is located on a finite set.
Similarly, if a compact smooth Riemannian manifold $M$ is embedded in $E=R^d$ and discretized by
a finite graph $G$, then the Euler curvature of the manifold can be obtained as the expectation of indices $i_f(x)$ where
$f$ is in a finite dimensional probability space of linear functions in the ambient space equipped with a probability
measure which is rotationally symmetric. The same probability space of functions $f$ produces random index functions
$i_f(x)$ on $G$. The expectation $K(x) = \E[ i_f(x)]$ is then believed to converge
to the classical curvature known in differential geometry. Index averaging becomes so a
``Monte Carlo" method for curvature and choosing different probability spaces allows to deform the geometry.
The measure can be obtained readily without any tensor analysis:
assume for example that the Riemannian manifold $M$ is diffeomorphic to a sphere $M_0= \phi(M)$
and Nash embedded in an Euclidean space, we get from the linear functions in the ambient space a probability measure $P$
on scalar functions of $M$. Taking the expectation of the index $i_f(x)$ gives curvature. 
Now push forward the measure $P$ to the function space $C^{\infty}(M_0)$ on $M_0$ to get a measure $\phi^* P$ on the 
set of all scalar functions on $M_0$. Averaging the index function $i_f(x)$ with respect to that measure gives the curvature function on the
standard sphere in such a way that $\phi: M \to M_0$ is an isometry - at least in the sense that curvature agrees.
We don't know yet that the Riemannian metric can be recovered from $\phi^* P$ but we believe it to be
possible since the probability measure $P$ gives more than just the expectation: it provides moments or
correlations and enough constraints to force the Riemannian metric.
This is an important question because if we would be able to measure distances from the measure $\phi^* P$,
geometry could be done by studying probability measures on the space of wave functions. 
Classical Riemannian manifolds are then part of a larger space of geometries in which
curvature $K$ is a distribution, but Gauss-Bonnet
still holds trivially as it is just the expectation of Poincar\'e-Hopf with respect to a measure. 
Any sequence of Riemannian metrics on a space $M$ had now an accumulation 
point in this bigger arena of geometries, as unlike the space of Riemannian metrics,
the extended space of measures is compact.

\begin{figure}
\parbox{15.2cm}{\scalebox{0.36}{\includegraphics{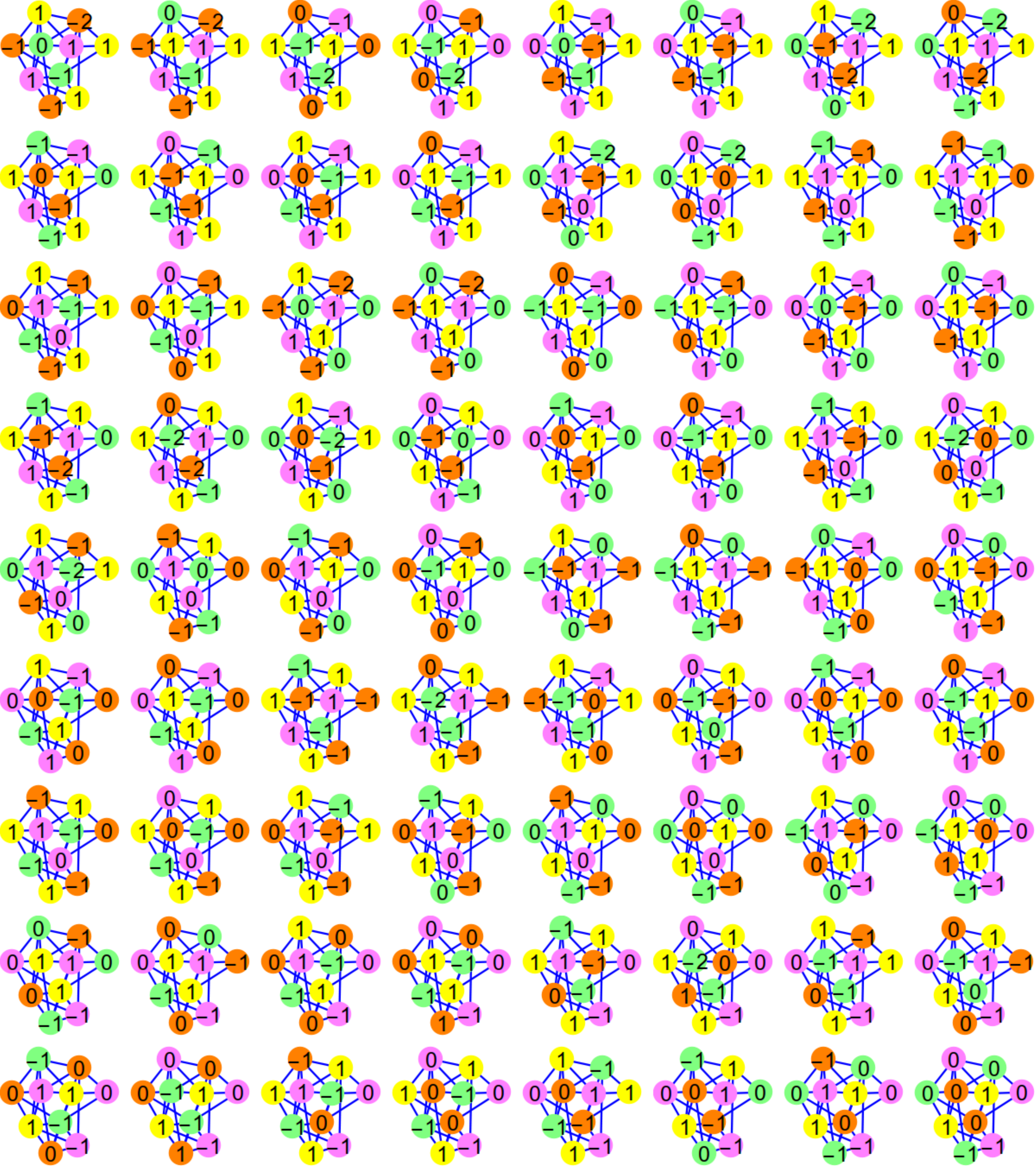}}}
\caption{
The figure shows a random graph with chromatic number $c=4$ and Euler characteristic $\chi=0$
and dimension $568/225=2.52444$.
The curvature $(-1/4, -1/4, 1/4, 1/6, -5/12, 1/12, 0, 1/6, 1/6, 1/12)$ is the
average over all 72 color index functions $i_f$ like for example
$i_f = \{ 1, -1, -1, 1, -2, -1, 1, 0, 1, 1 \}$
which has index $i_f(x)=-2$ at the vertex $x$ with minimal curvature $-5/12$. 
As the probability space has $C(c)=72$ elements, the 
color richness of this graph is $72/4! = 3$. 
}
\end{figure}

\vspace{12pt}
\bibliographystyle{plain}

\end{document}